\documentclass[11pt,leqno]{amsart}

\usepackage[margin=1.3in]{geometry}
\usepackage{amssymb}
\usepackage{amsmath,amscd}
\usepackage[initials,nobysame,alphabetic]{amsrefs}
\usepackage{amsmath, amssymb}
\usepackage{amsfonts}
\usepackage{mathrsfs}
\usepackage{mathpazo}
\usepackage[arrow,matrix,curve,cmtip,ps]{xy}

\usepackage{amsthm}
\usepackage{float}
\usepackage{hyperref}
\usepackage{graphics}
\usepackage{graphicx}
\usepackage{verbatim}
\usepackage{multirow}
\usepackage{tikz}
\usepackage{enumerate}
\usepackage{eufrak}
\allowdisplaybreaks

\newtheorem{theorem}{Theorem}[section]

\newtheorem{proposition}[theorem]{Proposition}

\newtheorem*{theorem*}{Theorem}
\newtheorem{remark}[theorem]{Remark}

\def\ae{\hbox{\rm a.e.{ }}}
\def\as{\hbox{\rm a.s.{ }}}

\newcommand\given[1][]{\:#1\vert\:}

\newcommand{\R}{\mathbb{R}}
\newcommand{\Ff}{\mathbb{F}}
\newcommand{\U}{\mathcal{U}}

\newcommand{\D}{\mathcal{D}}
\newcommand{\F}{\mathcal{F}}

\newcommand{\hH}{\mathcal{H}}

\newcommand{\Om}{\Omega}

\newcommand{\Prob}{\mathbb{\Prob}}

\newcommand{\mytilde}{\raise.17ex\hbox{$\scriptstyle\mathtt{\sim}$}}

\def\a{\alpha}             
\def\b{\beta}                                    
\def\s{\sigma}                    

\def\l{\lambda}
\def\k{\kappa}

\def\ms{\medskip} 
\def\no{\noindent}

\begin{document}
\title[An SMP for Markov chains of mean-field type]{A Stochastic Maximum Principle for Markov chains of mean-field type}

\author{Salah Eddine Choutri and Hamidou Tembine}

\address{Department of Mathematics \\ KTH Royal Institute of Technology \\ 100 44, Stockholm \\ Sweden}
\address{New York University, 19 Washington Square North New York, NY 10011, USA} 
\email{choutri@kth.se \quad tembine@nyu.edu}
\thanks{This research was supported by FILL IN }

\date{\today}

\subjclass[2010]{60H10, 60H07, 49N90}

\keywords{Mean-field, nonlinear Markov chain, Backward SDEs, optimal control, Stochastic maximum principle }

\begin{abstract}
We derive sufficient and necessary optimality conditions in terms of a stochastic maximum principle (SMP) for controls  associated with cost functionals of mean-field type, under a dynamics driven by a class of Markov chains of mean-field type which are pure jump processes obtained as solutions of a well-posed martingale problem.  As an illustration, we apply the result to generic examples of control problems as well as some applications.
\\

\end{abstract}

\maketitle

\tableofcontents

\section{Introduction}
The goal of this paper is to find sufficient and necessary optimality conditions in terms of a stochastic maximum principle (SMP) for a set of admissible controls $\hat{u}$, which minimize payoff functionals of the form
\begin{equation*}
J(u):=E^u\left[\int_0^T f(t,x,E^u[\k_f(x(t))],u(t))dt+ h\left(x(T),E^u[\k_h(x(T))]\right)\right],
\end{equation*}
w.r.t. admissible controls $u$, for some given functions $f, h, \k_f$ and $\k_h$, under dynamics driven by a pure jump process $x$ with state space $I=\{0,1,2,3,\ldots\}$ whose jump intensity under the probability measure $P^u$ is of the form 
$$
\l^u_{ij}(t):=\lambda_{ij}(t,x,E^u[\k(x(t))], u(t)),\,\,\quad i,j\in I,
$$
for some given functions $\l$ and $\k$, as long as the intensities are predictable. Due to the dependence of the intensities on the mean of (a function of) $x(t)$ under $P^u$,  the process $x$ is commonly called a nonlinear Markov chain or Markov chain of mean-type, although it is does not satisfy the standard Markov property as explained in the seminal paper by McKean \cite{McKean} for diffusion processes. A more general situation is when the jump intensities depend on the marginal law $P^u\circ x^{-1}(t)$ of $x(t)$ under $P^u$. To keep the content of the paper as simple as possible, we do not treat this general case. The dependence of the intensities on the whole path $x$ makes the jump process cover a large class of real-word applications.

The present work is a continuation of \cite{choutri2016} where we proved existence and uniqueness of this class of processes, in terms of a martingale problem, and derived sufficient conditions (cf. Theorem 4.6 in \cite{choutri2016}) for existence of an optimal control which minimizes $J(u)$, for a rather general class of (unbounded) jump intensities. Since the suggested conditions are rather difficult to apply in concrete situations (see Remark 4.7 and Example 4.8 in \cite{choutri2016}), we aim in this paper to investigate whether the SMP can yield optimality conditions that are tractable and easy to verify. 

While in the usual strong-type control problems, the dynamics is given in terms of a process $X^u$ which solves a stochastic differential equation (SDE) on {\it a given} probability space $(\Omega,\F,Q)$, the dynamics in our formulation is given in terms of a family of probability measures $(P^u,\,u\in\U)$ and $x$ as the coordinate process i.e. it does not change with the control $u$. This type of formulation is usually called weak-type formulation for control problems. 

The main idea in the Martingale and Dynamic Programming approaches to optimal control problems for jump processes (without mean-field coupling) suggested in previous work including the following first papers in the subject \cite{BV, Bis, DE, WD} (the list of references is far from being exhaustive), is to use the Radon-Nikodym density process $L^u$ of $P^u$ w.r.t. some reference probability measure $P$ as dynamics and recast the control problem to a standard one. In this paper we apply the same idea and recast the control problem to a mean-field-type control problem to which an SMP can applied. By a  Girsanov-type result for pure jump processes, the density process  $L^u$ is a martingale and solves a linear SDE driven by some accompanying $P$-martingale $M$. The adjoint process associated to the SMP solves a (Markov chain) backward stochastic differential equation (BSDE) driven by the $P$-martingale $M$, whose existence and uniqueness can be derived using the results by Cohen and Elliott \cite{Cohen1,Cohen2}. For some linear and quadratic cost functionals, we explicitly solve these BSDEs and derive a closed form of the optimal control. 

In Section 2, we briefly recall the basic stochastic calculus for pure jump processes we will use in the sequel. In Section 3, we derive sufficient and necessary optimality conditions for the control problem. The SMP optimality conditions are derived in terms  of a mean-field stochastic maximum principle where the adjoint equation is a Markov chain BSDE. In Section 3, we illustrate the results by two examples of optimal control problems that involve two-state chains and linear quadratic cost functionals. We also consider an optimal control of mean-field version of the Schl{\"o}gl model for chemical reactions. We consider linear and quadratic cost functionals in all examples for the sake of simplicity and  also because, in these cases, we obtain the optimal controls in closed form.

The obtained results can easily be extended to pure jump processes taking values on more general state spaces such as $I=\mathbb{Z}^d,\,d\ge 1$.


\section{Preliminaries}
Let $I:=\{0,1,2,\ldots\}$ equipped with its discrete topology and $\s$-field and let $\Om:=\D([0,T],I)$ be the space of functions from $[0,T]$ to $I$ that are right continuous with left limits at each $t\in [0,T)$ and are left continuous at time $T$. We endow $\Om$ with the Skorohod metric $d_0$ so that $(\Om,d_0)$ is a complete separable metric (i.e. Polish) space.  Given $t\in [0,T]$ and $\omega\in\Om$, put  $x(t,\omega)\equiv\omega(t)$ and denote by $\F^0_t:=\sigma(x(s),\,\, s\le t),\, 0\le t\le T,$ the filtration generated by $x$.  Denote by $\F$ the Borel $\sigma$-field over $\Om$. It is well known that $\F$ coincides with $\sigma( x(s),\,\, 0\le s\le T)$.  

To $x$ we associate the indicator process $I_i(t):=\mathbf{1}_{\{x(t)=i\}}$ whose value is $1$ if the chain is in state $i$ at time $t$ and $0$ otherwise, and the counting processes $N_{ij}(t),\,\,i\ne j$, independent of $x(0)$, such that 
$$
N_{ij}(t):=\#\{\tau\in(0,t]:x(\tau^-)=i, x(\tau)=j\},\quad N_{ij}(0)=0,
$$
which count the number of jumps from state $i$ into state $j$ during the time interval $(0,t]$. Obviously, since $x$ is right continuous with left limits, both $I_i$ and $N_{ij}$ are right continuous  with left limits. Moreover, by the relationship
\begin{equation}\label{x-rep-1}
x(t)=\sum_i iI_i(t),\quad I_i(t)=I_i(0)+\underset{j:\, j\neq i}\sum\left(N_{ji}(t)-N_{ij}(t)\right),
\end{equation}
the state process, the indicator processes, and the counting processes carry the same information which is represented by the natural filtration  $\Ff^0:=(\F^0_t,\, 0\le t\le T)$ of $x$. Note that \eqref{x-rep-1} is equivalent to the following useful representation
\begin{equation}\label{x-rep-2}
x(t)=x(0)+\sum_{i,j: \,i\neq j} (j-i) N_{ij}(t).
\end{equation}

Let $G=(g_{ij},\,i,j\in I)$, where $g_{ij}$ are constant entries,  be a $Q$-matrix:
\begin{equation}\label{G}
g_{ij}>0, \,\,\, i\neq j,\quad \underset{j:\, j\neq i}\sum g_{ij}<+\infty, \quad  g_{ii}=-\underset{j:\, j\neq i}\sum g_{ij}.
\end{equation}
By Theorem 4.7.3 in \cite{EK}, or Theorem 20.6 in \cite{RW} (for the finite state-space), given the $Q$-matrix $G$ and a probability measure $\xi$ over $I$,  there exists a unique probability measure $P$ on $(\Om,\F)$ under which the coordinate process $x$ is a time-homogeneous Markov chain with intensity matrix $G$ and starting distribution $\xi$  i.e. such that $P\circ x^{-1}(0)=\xi$. Equivalently, $P$ solves the martingale problem for $G$ with initial probability distribution $\xi$ meaning that,   
for every $f$ on $I$, the process defined by
\begin{equation}\label{f-mart-1}
M_t^f:=f(x(t))-f(x(0))-\int_{(0,t]}(Gf)(x(s))\,ds
\end{equation}
is a local martingale relative to $(\Om,\F,\Ff^0)$, where 
$$
Gf(i):=\sum_j g_{ij}f(j)=\sum_{j: \,j\neq i}g_{ij}(f(j)-f(i)),\,\,\, i\in I,
$$
and
\begin{equation}\label{G-f}
Gf(x(s))=\sum_{i,j: \,j\neq i}I_i(s)g_{ij}(f(j)-f(i)).
\end{equation}
By Lemma 21.13 in \cite{RW}, the compensated processes associated with the counting processes $N_{ij}$, defined by
 \begin{equation}\label{mart-1}
M_{ij}(t)=N_{ij}(t)-\int_{(0,t]} I_i(s^-)g_{ij}\, ds,\quad M_{ij}(0)=0,
\end{equation}
are zero mean, square integrable and mutually orthogonal $P$-martingales  whose  predictable quadratic variations are
\begin{equation}\label{mart-2}
\langle M_{ij}\rangle_t=\int_{(0,t]} I_i(s^-)g_{ij}\, ds.
\end{equation} 
Moreover, at jump times $t$, we have
\begin{equation}\label{mart-3}
\Delta M_{ij}(t)=\Delta N_{ij}(t)=I_i(t^-)I_j(t).
\end{equation}
Thus, the optional variation of $M$ 
$$
[M](t)=\sum_{0<s\le t}|\Delta M(s)|^2=\underset{0<s\le t}\sum\,\underset{i,j:\, j\neq i}\sum|\Delta M_{ij}(s)|^2
$$
is
\begin{equation}\label{optional-M}
[M](t)=\underset{0<s\le t}\sum\,\underset{i,j:\, j\neq i}\sum I_i(s^-)I_j(s).
\end{equation}
We call  $M:=\{M_{ij},\,\, i\neq j\}$ the accompanying martingale of the counting process $N:=\{N_{ij},\,\, i\neq j\}$ or of the Markov chain $x$. 

Denote by $\Ff:=(\F_t)_{0\le t\le T}$ the completion of $\Ff^0=(\F^0_t)_{t\le T}$ with the $P$-null sets of $\Omega$. Hereafter, a process from $[0,T]\times\Om$ into a measurable space is said predictable (resp. progressively measurable) if it is predictable (resp. progressively measurable) w.r.t. the predictable $\sigma$-field on $[0,T]\times \Om$ (resp. $\Ff$).

For a real-valued matrix $m:=(m_{ij},\, i,j \in I)$ indexed by $I\times I$, we let 
\begin{equation}\label{g-t}
\|m\|_g^2(t):=\underset{i,j:\, i\neq j}\sum |m_{ij}|^2g_{ij}\mathbf{1}_{\{x(t^-)=i\}}<\infty.
\end{equation}
If $m$ is time-dependent, we simply write $\|m(t)\|_g^2$.

\section{A stochastic maximum principle}

We consider controls with values in some subset $U$ of $\R^d$ and let $\U$ be the set of $\Ff$-progressively measurable processes $u=(u(t),\,0\le t\le T)$ with values in $U\subset \R^d$. $\U$ is the set of admissible controls. 

\ms\no
For $u\in\U$, let $P^u$ be the probability measure on $(\Om,\F)$ under which the coordinate process $x$ is a jump process with intensities 
\begin{equation}\label{u-lambda}
\l_{ij}^u(t):=\l_{ij}(t,x,E^u[\k(x(t)]),u(t)),\,\,\, i,j\in I,\,\, 0\le t\le T,
\end{equation}
The cost functional  associated to $P^u$ is of the form
\begin{equation}\label{J-u}
J(u):=E^u\left[\int_0^T f(t,x,E^u[\k_f(x(t))],u(t))dt+ h\left(x(T),E^u[\k_h(x(T))]\right)\right].
\end{equation}

\ms
In this section we propose to characterize minimizers $\bar u$ of $J$ i.e. $\bar u\in\U$ satisfying
\begin{equation}\label{opt-J}
J(\bar u)=\min_{u\in\U}J(u)
\end{equation}
in terms of a stochastic maximum principle (SMP). We first state and prove the sufficient optimality conditions. Then, we state the necessary optimality conditions.

\ms Let $P$  be the probability measure on $(\Omega, \mathcal F)$ under which $x$ is a time-homogeneous Markov chain such that $P\circ x^{-1}(0)=\xi$ and with $Q$-matrix $(g_{ij})_{ij}$  satisfying \eqref{G}.
Then, by a Girsanov-type result for pure jump processes (see e.g. \cite{RW, bremaud}), it holds that
 \begin{equation}\label{P-u}
 dP^u:=L^u(T) dP,
 \end{equation}
where, for $0\le t\le T$,
\begin{equation}\label{P-u-density}
 L^u(t):=\underset{\substack{i,j\\ i\neq j} }\prod \exp{\left\{ \int_{(0,t]}\ln{\frac{ \l_{ij}^u(s)}{g_{ij}}}dN_{ij}(s)-\int_0^t( \l_{ij}^u(s)-g_{ij})I_i(s)ds  \right\}}, 
\end{equation}
 which satisfies
\begin{equation}\label{L-u-1}
L^u(t)=1+\int_{(0,t]} L^u(s^-)\underset{i,j:\, i\neq j}\sum I_i(s^-)\ell^u_{ij}(s)dM_{ij}(s),
\end{equation}
where $\ell^u_{ij}(s):=\ell_{ij}(t, x,E^u[\k(x(s))],u(s))$ is given by the formula
\begin{equation}\label{L-u-2}
\ell^u_{ij}(s)=\left\{\begin{array}{rl}
\l_{ij}^u(s)/g_{ij}-1 &\text{if }\,\, i\neq j,\\ 0 & \text{if }\,\, i=j,
\end{array}
\right.
\end{equation}
and $(M_{ij})_{ij}$ is the $P$-martingale given in \eqref{mart-1}. Moreover, the accompanying martingale $M^u=(M^u_{ij})_{ij}$ satisfies
\begin{equation}\label{m-u-1}
M^u_{ij}(t)=M_{ij}(t)-\int_{(0,t]} \ell^u_{ij}(s)I_i(s^-)g_{ij}ds.
\end{equation}

\ms\no Noting that 
$$
J(u)=E\left[L^u(T)\int_0^T f(t,x,E^u[\k_{f}(x(t))],u(t))dt+ L^u(T) h(x(T),E^u[\k_{h}(x(T))])\right].
$$
Integrating by parts and taking expectation, we obtain
{\small 
\begin{equation}\label{J-u-mp-1}
J(u):=E\left[\int_0^T L^u(t)f(t,x,E[L^u(t)\k_{f}(x(t))],u(t))dt+ L^u(T)h(x(T),E[L^u(T)\k_{h}(x(T))])\right].
\end{equation}
}

\ms\no We have recast our problem of controlling a Markov chain through its intensity matrix to a standard control problem which aims at minimizing the cost functional \eqref{J-u-mp-1} under the dynamics given by the density process $L^u$ which satisfies \eqref{L-u-1}, to which the mean-field stochastic maximum principle in \cite{ Buckdahn2} can be applied.  The corresponding optimal dynamics is given by the probability measure $\bar{P}$ on $(\Om,\F)$ defined by
\begin{equation}\label{opt-P}
d\bar{P}=L^{\bar u}(T)dP,
\end{equation}
where $L^{\bar u}$ is the associated density process. $(L^{\bar u}, \bar{u})$ is called optimal pair associated with \eqref{opt-J}.   

\ms\no
For $w=y,\bar y,u$, $\psi_{w}$ denotes the partial derivative of the function $\psi(y,\bar y,u)$ w.r.t.  $w$.

\ms\no for $\a=\ell, f,h$, we set 
$$
\a(t):=\a(t,E[L^{u}(t)\k_{\a}(x(t))],u(t)),\,\,\, \bar\a(t):=\a(t,E[L^{\bar u}(t)\k_{\a}(x(t))],\bar u(t)).
$$

\ms To the admissible pair of processes $(L^{\bar u},\bar u)$ we associate the solution $(p,q)$ (if it exits) of the following linear BSDE of mean-field type, known as first-order adjoint equation:
\begin{equation}\label{mp-pq-adjoint}
\left\{ \begin{array}{lll}
dp(t)=-\left\{\langle \bar\ell(t),q(t)\rangle_{g}- \bar f(t)+\k(x(t))E[L^{\bar u}(t)(\langle\bar\ell_{\bar y}(t),q(t)\rangle_{g}]\right. \\ \left. \qquad\qquad\qquad\qquad -\k_{f}(x(t))E[L^{\bar u}(t)\bar f_{\bar y}(t)\right\}dt +q(t)dM(t),\\ \\ p(T)=-\bar h(T)-\k_{h}(x(T))E[L^{\bar u}(T)\bar h_{\bar y}(T)].
\end{array}
\right.
\end{equation}

\ms In the next proposition we give sufficient conditions on $f, h, \ell, \k, \k_{f}$ and $\k_{h}$ that guarantee existence of a uniques solution to the BSDE \eqref{mp-pq-adjoint}.  
\begin{proposition}\label{mp-pq}
Assume that
\begin{enumerate}
\item[(A1)] The function $\ell$ is differentiable in $\bar y$. Moreover,  $\|\ell\|_g$ and $\|\ell_{\bar y}\|_g$ are bounded. 
\item [(A2)] The functions $f, h,\k_f$ and $\k_{h}$ are bounded.  $f$ and $h$ are differentiable in $\bar y$ with bounded derivatives.
\end{enumerate}
Then, the BSDE \eqref{mp-pq-adjoint}  admits a solution $(p,q)$ consisting of an adapted process $p$ which is right-continuous with left limits and a predictable process $q$ which satisfy  
\begin{equation}\label{mp-pq-estim-1}
E\left[\sup_{t\in[0,T]}|p(t)|+\int_{(0,T]} \|q(s)\|^2_{g}ds\right]<+\infty.
\end{equation}  
This solution is unique up to indistinguishability for $p$ and equality  
$dP\times g_{ij}I_i(s^-)ds$-almost everywhere for $q$. 
\end{proposition}

\begin{proof} Assumptions (A1) and (A2) make the driver of the BSDE \eqref{mp-pq-adjoint}  Lipschitz continuous in $q$. The proof is similar to that of Theorem 3.1 for the Brownian motion driven mean-field BSDE derived in  \cite{Buckdahn} by considering the following norm 
$$
\|(p,q)\|^2_{\beta}:=E\int_0^T e^{\beta t}(|p(t)|^2+\|q(t)\|^2_g)dt,
$$
where $\beta >0$, along with It\^o's formula for purely discontinuous semimartingales. We omit the details.
\end{proof}

\begin{remark} 
\begin{itemize}
\item [(i)] The boundedness on $f$ and $h$ and their derivatives  is strong and can be considerably weakened using standard truncation techniques. 

\item [(ii)] If $\ell_{\bar y}= 0$ i.e. the intensity does not contain any mean-field coupling, the BSDE \eqref{mp-pq-adjoint}   becomes standard. Thanks to Theorem 3.10 in \cite{choutri2016}, it is solvable only by imposing similar conditions to (H1)-(H3) therein.

\item [(iii)] If $\ell_{\bar y}\neq 0$ i.e. the intensity is of mean-field type,  we don't know whether we can relax the imposed boundedness of $\ell, \k$ and $\ell_{\bar y}$ because without this condition the standard comparison theorem  for Markov chain BSDEs simply does not, in general, apply for such drivers. 
\end{itemize}
\end{remark}

Let $(L^{\bar u},\bar u)$ be an admissible pair and $(p,q)$ be the associated first order adjoint process solution of \eqref{mp-pq-adjoint}. 
 
\ms For $v\in U$, we introduce the Hamiltonian associated to our control problem 
\begin{equation}\label{ham-u}
H(t,v):=L^{\bar u}(t)\left(\langle\ell(t,E[L^{\bar u}(t)\k(x(t))],v),q(t)\rangle_{g}-f(t,E[L^{\bar u}(t)\k_{f}(x(t))],v)\right).
\end{equation}

Next, we state the SMP sufficient and necessary optimality conditions, but only prove the sufficient optimality case, as the necessary optimality  conditions result is tedious and more involved but by now 'standard' and can be derived following the same steps of \cite{ Buckdahn2, Shen, Tang}.  

\ms
In the next two theorems we assume that (A1) and (A2) of Proposition \eqref{mp-pq} hold.

\begin{theorem}[Sufficient optimality conditions]\label{sufficient-mp}
 Let $(L^{\bar u},\bar u)$ be an admissible pair and $(p,q)$ be the associated first order adjoint process which satisfies \eqref{mp-pq-adjoint}-\eqref{mp-pq-estim-1}.  Assume
 \begin{itemize}
\item[(A4)] The set of controls $U$ is a convex body (i.e. $U$ is convex and has a nonempty interior) of $\R^d$, and the functions $\ell$ and $f$ are  differentiable in $u$.
\item[(A5)]  The functions $(y,\bar y,u)\mapsto y\ell(\cdot,\bar y,u)$ and $(y,\bar y,u)\mapsto -yf(\cdot,\bar y,u)$ are concave in $(y,\bar y,u)$  for  $\ae\, t\in [0,T]$, $P$-almost surely,
\item[(A6)] The function $(y,\bar y)\mapsto yh(\cdot,\bar y)$ is convex .
\end{itemize}
If the admissible control $\bar u$ satisfies
\begin{equation}\label{H-max}
H(t,\bar u(t))=\max_{v\in U} H(t,v), \qquad \ae t\in[0,T],\quad  P\text{-}\as
\end{equation}
then, the pair $(L^{\bar u},\bar u)$ is optimal.
\end{theorem}
\begin{proof}  
We want to show that if the pair $(L^{\bar u},\bar u)$ satisfies \eqref{H-max}, then
$$
J(u)-J(\bar u)=E\left[ \int_0^T (L^{u}(t)f(t)-L^{\bar u}(t)\bar f(t))dt+L^{u}(T)h(T)-L^{\bar u}(T)\bar h(T)\right]\ge 0.
$$
Since  $(y,\bar y)\mapsto yh(\cdot,\bar y)$  is  convex, we have
$$
\begin{array}{lll}
E[L^{u}(T)h(T)-L^{\bar u}(T)\bar h(T)]\ge E[( \bar h(T)+\k(T)E[L^{\bar u}(T)\bar h_{\bar y}(T)])(L^{u}(T)-L^{\bar u}(T))]\\ \qquad\qquad\qquad\qquad\qquad\quad =-E[p(T)(L^{u}(T)-L^{\bar u}(T))]. 
\end{array}
$$
Integrating by parts, using \eqref{mp-pq-adjoint}, we obtain
$$\begin{array}{lll}
E[p(T)(L^{u}(T)-L^{\bar u}(T))]=E\left[\int_0^T (L^{u}(t^-)-L^{\bar u}(t^-))dp(t)+p(t^-) d(L^{u}(t)-L^{\bar u}(t)) \right.\\ \left. \qquad\qquad\qquad\qquad\qquad\qquad\qquad +d[L^{u}-L^{\bar u},p](t)\right]\\ \qquad\qquad\qquad\qquad\qquad
=-E\left[\int_0^T \left(\left\{\langle \bar\ell(t),q(t)\rangle_{g}- \bar f(t)+\k(x(t))E[L^{\bar u}(t)(\langle\bar\ell_{\bar y}(t),q(t)\rangle_{g}] \right.\right. \right.\\ \left. \left. \left. \qquad\qquad\qquad\quad -\k_{f}(x(t))E[L^{\bar u}(t)\bar f_{\bar y}(t)\right\}(L^{u}(t)-L^{\bar u}(t))-\langle L^{u}(t)\ell(t)-L^{\bar u}(t)\bar\ell(t),q(t)\rangle_g\right)dt\right].
\end{array}
$$
We introduce the following 'Hamiltonian' function:
\begin{equation}\label{mp-hamiltonian}
\hH(t,y,\bar y,u,z):=y\langle\ell(t,\bar y,u),z\rangle_{g}-yf(t,\bar y,u).
\end{equation}
Furthermore, for $u$ and $\bar{u}$ in $\U$, we set 
\begin{equation}\label{mp-hamiltonian-1} \left\{ \begin{array}{lll}
\hH(t):=L^{u}(t)\left(\langle\ell(t,E[L^{u}(t)\k(x(t))],u(t)),q(t)\rangle_{g}-f(t,E[L^{u}(t)\k_{f}(x(t))],u(t))\right), \\

\bar{\hH}(t):=L^{\bar u}(t)\left(\langle\ell(t,E[L^{\bar u}(t)\k(x(t))],\bar{u}(t)),q(t)\rangle_{g}-f(t,E[L^{\bar u}(t)\k_{f}(x(t))],\bar{u}(t))\right).
\end{array}
\right.
\end{equation}

Since $(y,\bar y,u)\mapsto y\ell(\cdot,\bar y,u)$ and $-yf(\cdot,\bar y,u)$ are concave,  we have
$$
\hH(t)-\bar \hH(t)\le \bar \hH_{y}(t)(L^{u}(t)- L^{\bar u}(t))+\bar \hH_{\bar y}(t)(E[\k(x(t))(L^{u}(t)-L^{\bar u}(t))])+\bar \hH_u(t)\cdot(u(t)-\bar u(t)).
$$
Since, by \eqref{H-max}, $\bar{\hH}_u(t)=H_u(t,\bar{u}(t))=0 \,\,\, \ae t\in[0,T]$, we obtain
\begin{equation*}\label{5-concave}\begin{array}{ll}
E[\hH(t)-\bar\hH(t)]\le E\left[\left\{\langle \bar\ell(t),q(t)\rangle_{g}- \bar f(t)+\k(x(t))E[L^{\bar u}(t)(\langle\bar\ell_{\bar y}(t),q(t)\rangle_{g}]
\right.\right. \\ \left. \left. \qquad\qquad\qquad\qquad\qquad\qquad  -\k_{f}(x(t))E[L^{\bar u}(t)\bar f_{\bar y}(t)\right\}(L^{u}(t)-L^{\bar u}(t))\right]
\end{array}
\end{equation*}
for $\ae t\in[0,T]$.

Therefore, 
\begin{equation*}\label{h-convex}
E[L^u(T)h(T)-L^{\bar u}(T)\bar{h}(T)]\ge E\left[\int_0^T \left(\hH(t)-\bar \hH(t)-\langle L^{u}(t)\ell(t)-L^{\bar u}(t)\bar\ell(t),q(t)\rangle_g\right)dt\right].
\end{equation*}
Hence,
$$\begin{array}{lll}
J(u)-J(\bar u)\ge E\left[\int_0^T \left(\hH(t)-\bar \hH(t)+L^{u}(t)f(t)-L^{\bar u}(t)\bar f(t) \right.\right. \\ \left.\left. \qquad\qquad\qquad\qquad\qquad\qquad\qquad\qquad -\langle L^{u}(t)\ell(t)-L^{\bar u}(t)\bar\ell(t),q(t)\rangle_g\right)dt\right]=0.
\end{array}
$$
\end{proof}

\begin{theorem}[Necessary optimality conditions (Verification Theorem)]\label{necessary-mp}
If $(L^{\bar{u}}, \bar{u})$ is an optimal pair of the control problem  \eqref{opt-J} and 
there is a unique pair of $\mathcal{F}$-adapted processes $(p, q)$, associated to $(L^{\bar{u}}, \bar{u})$, which satisfies \eqref{mp-pq-adjoint}-\eqref{mp-pq-estim-1}, then 
\begin{equation*}
H(t,\bar u(t))=\max_{v\in U} H(t,v), \qquad \ae t\in[0,T],\quad  P\text{-}\as
\end{equation*}
\end{theorem}

\begin{remark}\label{feasibility}
Unfortunately, the sufficient optimality conditions are almost rare to obtain, due to the fact that the convexity conditions imposed on the involved coefficients are not always satisfied, even for the simplest examples: assume $\ell$ and $f$ without mean-field coupling and linear in the control $u$. Then none of the functions $(y,u)\mapsto y\ell(\cdot,u)$ and $(y,u)\mapsto -yf(\cdot,u)$ is concave in $(y,u)$. However, the verification theorem in terms of necessarily optimality conditions holds for a fairly general class of functions with sufficient smoothness. Hence, if we can solve the associated BSDEs, the necessary optimality conditions result can be useful.
\end{remark}

\section{Examples } In this section we first  solve the adjoint equation associated to an optimal control problem associated with a standard two-state Markov chain, then we extend the problem to a two-state Markov chain of man-field type. As mentioned in Remark \eqref{feasibility}, whether sufficient or necessary conditions may apply depends of course on the smoothness of the involved functions. Not all the functions involved in the next examples satisfy the convexity conditions imposed in Theorem \eqref{sufficient-mp}.

\subsection*{Example 1. Optimal control of a standard two-state Markov chain}
We study the optimal control of a simple Markov chain $ x$ whose state space is $ \mathcal {X}=\{a,b\}$, where $(0\le a<b)$ are integers,  and its jump intensity matrix  is
\[
   \lambda^u(t)=
  \left[ {\begin{array}{cc}
   -\a & \a \\
   u(t) & -u(t) \\
  \end{array} } \right],
\]

\no where $\a $ is a given positive constant intensity and $u$ is the control process we assume  nonnegative, bounded and predictable. 
Let $P$ the probability measure under which the chain $x$  has intensity matrix 
\[
   G=
  \left[ {\begin{array}{cc}
   -g_{ab} & g_{ab} \\
   g_{ba} & -g_{ba} \\
  \end{array} } \right], \quad g_{ab},\,\,g_{ba} > 0. 
\]
Further, let $L^u(t)=\frac{dP^u}{dP}\big|_{\F_t}$  be the density process given by \eqref{L-u-1}, where  $\ell$ is defined by 
\begin{equation*}\label{L-u-2-mp}
\ell^{u}_{ij}(t)=\left\{\begin{array}{rl}
\l^{u}_{ij}(t)/g_{ij}-1 &\textbf{if }\,\, i\neq j,\\ 0 & \textbf{if }\,\, i=j.
\end{array}
\right.
\end{equation*}

The control problem we want to solve consists of finding the optimal control $\bar{u}$ that minimizes the linear-quadratic cost functional
\begin{equation}\label{cost-no-mean}
J(u)=E^u\left[\frac{1}{2}\int_0^T u^2(t)dt+h(x(T))\right], \quad h(b) \ge h(a).
\end{equation}

\no Given a control $v\in U$, consider the Hamiltonian 
$$ H(t,L^{\bar{u}}(t),q(t),v):= L^{\bar{u}}(t) ( \langle \ell^{v}(t) ,q(t)\rangle_{g} -\frac{1}{2}v^2 )=:H(t,v),$$
where 
$$
\langle \ell^{v}(t) ,q(t)\rangle_{g} =q_{ab}(t) (\a-g_{ab}) I_a(t^-)+q_{ba}(t) (v-g_{ba}) I_b(t^-).$$ 

\no By the first order optimality conditions, an optimal control $\bar u$ is solution of  the equation $\frac{\partial H(t,v)}{\partial v}=0$, which implies
\begin{align*}
0 = \left\langle  \frac{\partial \ell^v}{\partial v}(t),q(t)  \right\rangle_g-v =q_{ba}(t)I_b(t^-)-v. 
  \end{align*}
The optimal control is thus
\begin{equation}\label{ex-1-opt-u}
\bar{u}(t)=q_{ba}(t)I_b(t^-).
\end{equation}
where, for each $t$, $q_{ba} (t)\geq 0 $, since  $\bar{u} (t)\geq 0$. \\
It remains to identify $q_{ba}(t)$. Consider the associated adjoint equations given by
\begin{equation*}
\left\{ \begin{array}{lll}
dp(t) =-\left\{ \langle \ell^{\bar{u}}(t),q(t)  \rangle_g-\frac{1}{2}\bar{u}^2(t)  \right\} dt +q_{ab}(t)dM_{ab}(t)+q_{ba}(t)dM_{ba}(t) ,\quad 0\le t<T,\\ \\ 
p(T)= -h(x(T)).
\end{array}
\right.
\end{equation*}
In view of \eqref{ex-1-opt-u}, the driver reads
 \begin{equation}\label{driver}
 \begin{array}{lll}
\langle \ell^{\bar{u}}(t), q(t) \rangle_g-\frac{1}{2}\bar{u}^2(t) =q_{ab}(t) I_a(t^-)(\a-g_{ab}) +q_{ba}(t) I_b(t^-) \left\{ \frac{1}{2} q_{ba}(t) I_b(t^-)- g_{ba}  \right\}. 
\end{array}
\end{equation}
The adjoint equation becomes 
$$
dp(t)= q_{ab}(t) \left\{ -(\a-g_{ab}) I_a(t^-) dt+ dM_{ab}(t)\right\}+ q_{ba}(t) \left\{ -(\frac{1}{2}q_{ba}(t)-g_{ba}) I_b(t^-) dt+ dM_{ba}(t)\right\}.
$$ 
Now, considering the probability measure $\widetilde{P}$ under which  $x$ is a Markov chain whose jump intensity matrix
\[
  \widetilde{G}(t)=
  \left[ {\begin{array}{cc}
   - \a & \a \\
  \frac{1}{2} q_{ba}(t) & -\frac{1}{2} q_{ab}(t) \\
  \end{array} } \right],
  \] 
 the processes defined by 
\begin{equation*}
\left\{ \begin{array}{lll}
d\widetilde{M}_{ab}(t)=dM_{ab}(t)- (\a-g_{ab})I_a(t^-) dt, \\ 
d\widetilde{M}_{ba}(t)=dM_{ba}(t)- (\frac{1}{2}q_{ba}(t)-g_{ba}) I_b(t^-) dt,
\end{array}
\right.
\end{equation*}
are $\widetilde{P}$-martingales having the same jumps as the martingales $M_{ij}$:
\begin{equation}\label{ex-1-M-jump}
\Delta\widetilde{M}_{ab}(t)=\Delta M_{ab}(t)=I_a(t^-)I_b(t),\,\, \Delta\widetilde{M}_{ba}(t)=\Delta M_{ba}(t)=I_b(t^-)I_a(t)
\end{equation}
and  
\begin{equation}\label{ex-1-p}
dp(t)= q_{ab}(t) d\widetilde{M}_{ab}(t)+ q_{ba}(t) d\widetilde{M}_{ba}(t).
\end{equation}
This yields 
\begin{equation}\label{ex-1-delta-p}
\Delta p(t)=q_{ab}(t)I_a(t^-)I_b(t)+q_{ba}(t)I_b(t^-)I_a(t).
\end{equation}
Integrating \eqref{ex-1-p} and then taking conditional expectation yields
$$
p(t) = -\widetilde{E} [h(x(T))| \mathcal{F}_t ].
$$
Therefore, 
\begin{equation} \label{jump}
\Delta p(t)= -\Delta\widetilde{E} [h(x(T))| \mathcal{F}_t ].
\end{equation}
Under the probability measure $\widetilde{P} $
\begin{align*}
h(x(T))&=h(x(t))+ \int_t^T \left\{ \a (h(b)-h(a)) I_a(s^-)+ \frac{1}{2}q_{ba}(s) (h(a)-h(b))  I_b(s^-)\right\} ds \\ &+ \int_t^T   (h(b)-h(a))d\widetilde{M}_{ab} (s) +   \int_t^T  (h(a)-h(b))d\widetilde{M}_{ba}(s).
\end{align*}

\noindent Taking conditional expectation, we obtain
\begin{align*}
\widetilde{E} [h(x(T))| \mathcal{F}_t ] = h(x(t))+  \int_t^T \widetilde{E} \left[ \a(h(b)-h(a)) I_a(s^-)+ \frac{1}{2}q_{ba}(s) (h(a)-h(b))  I_b(s^-) | \mathcal{F}_t \right] ds,
\end{align*}
and 
\begin{align*}
\Delta \widetilde{E} [h(x(T))| \mathcal{F}_t ] = \Delta h(x(t))= -(h(b)-h(a)) I_a(t^-) I_b(t)-(h(a)-h(b)) I_b(t^-) I_a(t),
\end{align*}
which in view of  \eqref{jump} implies that
\begin{equation}\label{ex-1-q}
q_{ab}(t) =h(a)-h(b), \,\,\ q_{ba}(t) =h(b)-h(a).
\end{equation}
Therefore,
\begin{align*}
\bar{u}(t) = (h(b)-h(a)) I_b(t^-)= h(b) I_b(t^-)-h(a)+h(a)I_a(t^-)=h(x(t^-))-h(a),
\end{align*}
which yields the following explicit form of the optimal control:
$$
\bar{u}(t)=h(x(t^-))-h(a).
$$

\ms
In the next two examples we highlight  the effect of the mean-field coupling in both the jump intensity and the cost functional on the optimal control.

\subsection*{Example 2. Mean-field optimal control of a two-state Markov chain}   We consider the same chain as in the first example but with the following mean-field type jump intensities, $(t\in [0,T])$,

\[
   \lambda^u(t)=
  \left[ {\begin{array}{cc}
   -\a & \a \\
   u(t)+E^u[x(t^-)] & -u(t)-E^u[x(t^-)] \\
  \end{array} } \right],\quad \a>0,\quad u(t)+E^u[x(t^-)]\ge 0.
\]

and want to minimize the cost functional 

\begin{equation}\label{cost-mean}
J(u)=E^u\left[\frac{1}{2}\int_0^T u^2(t)dt\right]+Var^u(x(T)),
\end{equation}
where $Var^u(x(T))$ denotes the variance of $x(T)$ under the probability $P^u$ defined by
$$
Var^u(x(T)):=E^u\left[\left(x(T)-E^u[x(T)]\right)^2\right].
$$
\no Given a control $v\in U$, consider the Hamiltonian 
$$ H(t,v):= L^{\bar{u}}(t) ( \langle \ell^{v}(t) ,q(t)\rangle_{g} -\frac{1}{2}v^2 ),$$
where 
$$\langle \ell^{v}(t) ,q(t)\rangle_{g} =q_{ab}(t) (\a-g_{ab}) I_a(t^-)+q_{ba}(t) (v+E^{\bar{u}}[x(t^-)]-g_{ba}) I_b(t^-).$$ 

\no Performing similar calculations as in Example 1, we find that the optimal control is given by 
\begin{equation}\label{ex-2-opt-u}
\bar{u}(t)=q_{ba}(t) I_b(t^-).
\end{equation}

\no We will now identify $q_{ba}$. The associated adjoint equation is given by 

\begin{equation*}
\left\{ \begin{array}{lll}
dp(t)=-\left\{ \langle \ell^{\bar{u}}(t),q(t)  \rangle_g-\frac{1}{2}\bar{u}^2(t) +x(t) E^{\bar{u}}[\bar{H}_{\bar{y}}(t)]\right\} dt +q_{ab}(t)dM_{ab}(t)+q_{ba}(t)dM_{ba}(t) ,\\ \\ 
p(T)= -\left(x(T)-E^{\bar{u}}[x(T)]\right)^2.
\end{array}
\right.
\end{equation*}
In view of \eqref{ex-2-opt-u}, the driver reads
\begin{equation*}\begin{array}{lll}
\langle \ell^{\bar{u}}(t),q(t)  \rangle_g-\frac{1}{2}\bar{u}^2(t)+x(t) E^{\bar{u}}[\bar{H}_{\bar{y}}(t)] = 
q_{ba}(t) \left(\frac{1}{2} q_{ba}(t)+E^{\bar{u}}[x(t^-)]-g_{ba} \right) I_b(t^-) \\ \qquad\qquad\qquad\qquad\qquad\qquad\qquad  
+q_{ab}(t) (\a-g_{ab}) I_a(t^-)+x(t) E^{\bar{u}}[q_{ba}(t)I_b(t^-)]
\end{array}
\end{equation*}
The adjoint equation becomes
\begin{equation}\begin{array}{lll}
dp(t)=q_{ab}(t) \left\{dM_{ab}(t)  -(\a-g_{ab}) I_a(t^-)dt\right\}-x(t) E^{\bar{u}}[q_{ba}(t)I_b(t^-) ] dt\\ \qquad\qquad\qquad\qquad +q_{ba}(t)\left\{ dM_{ba}(t)- \left( \frac{1}{2}q_{ba}(t)+E^{\bar{u}}[x(t^-)]-g_{ba} \right) I_b(t^-) dt \right\}.
\end{array}
\end{equation}

\no Consider the probability measure $\widetilde P $, under which $x$ is a Markov chain whose jump intensity matrix

\[
   \widetilde G (t)=
  \left[ {\begin{array}{cc}
   -\a & \a \\
   \frac{1}{2} q_{ba}(t)+E^{\bar{u}}[x(t^-)] & -\frac{1}{2} q_{ba}(t)-E^{\bar{u}}[x(t^-)] \\
  \end{array} } \right], \quad  \frac{1}{2} q_{ba}(t)+E^{\bar{u}}[x(t^-)] \geq 0.
\]

\no This change of measure yields the $\widetilde P-$martingales

\begin{equation*}
\left\{ \begin{array}{lll}
d\widetilde{M}_{ab}(t)=dM_{ab}(t)- (\a-g_{ab})I_a(t^-) dt, \\ 
d\widetilde{M}_{ba}(t)=dM_{ba}(t)- (\frac{1}{2}q_{ba}(t)+E^{\bar{u}}[x(t^-)] -g_{ba}) I_b(t^-) dt
\end{array}
\right.
\end{equation*}

\no and 
\begin{equation} \label{adjoint-ex2}
dp(t)=- x(t) E^{\bar{u}}[q_{ba}(t)I_b(t^-) ] dt+ q_{ab}(t)d\widetilde{M}_{ab}(t)+ q_{ba}(t) d\widetilde{M}_{ba}(t).
\end{equation}

\no This yields 
\begin{equation} \label{delta p-ex02}
\Delta p(t)  = q_{ab} (t) I_a(t^-) I_b(t)+ q_{ba}(t) I_b(t^-)I_a(t).
\end{equation}

\no Integrating \eqref{adjoint-ex2}, then taking conditional expectation yields

$$
p(t) = - \widetilde{E} [ \left(x(T)-E^{\bar{u}}[x(T)]\right)^2 | \mathcal{F}_t ] + \widetilde{E} [ \int_t^T x(s) E^{\bar{u}}[q_{ba}(t)I_b(s^-) ] ds    | \mathcal{F}_t ]  
$$

\no Therefore,
\begin{align} \label{delta-phi}
\Delta p(t)= -\Delta \widetilde{E} [\left(x(T)-E^{\bar{u}}[x(T)]\right)^2| \mathcal{F}_t ].
\end{align}
\no Next, we compute the right hand side of \eqref{delta-phi}, then we identify $q_{ba}$ by matching. 
\\ Set $\bar{\mu}(t):= E^{\bar{u}} [x(t)]$and $ \phi(t,x(t)):=  (x(t)-\bar{\mu}(t))^2. $ Under $\widetilde P,$ Dynkin's formula yields
$$\phi(T,x(T))= \phi(t,x(t))+ \int_t^T \left( \frac{\partial \phi}{\partial s} +\widetilde G \phi\right)(s,x(s)) ds + \widetilde M_T^{\phi}- \widetilde M_t^{\phi} . $$

\no Taking conditional expectation yields
$$ \widetilde E [\phi(T,x(T))|\mathcal{F}_t ] = \phi(t,x(t))+ \int_t^T  \widetilde E \left[  \left( \frac{\partial \phi}{\partial s} +\widetilde G \phi\right)(s,x(s)) \given[\Big] \mathcal{F}_t \right] ds, $$

\no and
\begin{align*}
\Delta \widetilde E [\phi(T,x(T))|\mathcal{F}_t ]  = \Delta \phi (t,x(t)),
\end{align*}
 \no where
 \begin{align*}
 \Delta \phi (t,x(t)) &= \underset{i,j : \ i \neq j }\sum \left(\phi(t,j)-\phi(t,i) \right) I_i(t^-)I_j(t)
  \quad \quad (i,j \ \in \{a,b\}) \\        
                     &= \left( (b^2-a^2)-2\bar{\mu}(t)(b-a) \right) I_a(t^-) I_b(t)+ \left((a^2-b^2)-2\bar{\mu}(t)(a-b)\right) I_b(t^-) I_a(t).
 \end{align*}
Therefore, 
\begin{align}
\Delta p(t) &= - \Delta \widetilde{E} [\left(x(T)-\bar{\mu}(T) \right)^2| \mathcal{F}_t ] \nonumber \\
            &= \left( (a^2-b^2)+2\bar{\mu}(t)(b-a) \right) I_a(t^-) I_b(t)+ \left((b^2-a^2)+2\bar{\mu}(t)(a-b)\right) I_b(t^-) I_a(t).\label{matching-ex2}
\end{align}

\no  Matching \eqref{delta p-ex02} with \eqref{matching-ex2} yields

\begin{equation*}
\left\{ \begin{array}{lll}
q_{ab}(t)  = (a^2-b^2)+2\bar{\mu}(t)(b-a),   \\ 
q_{ba}(t)  = (b^2-a^2)+2\bar{\mu}(t)(a-b).
\end{array}
\right.
\end{equation*}

\no Hence, 
$$
\bar{u}(t) =\left( (b^2-a^2)+2\bar{\mu}(t)(a-b) \right) I_b(t^-). 
$$

Noting that $a\le \bar\mu(t)\le b$, to  guarantee that both $\lambda^{\bar u}(t)$ and $\widetilde G(t)$ above are indeed intensity matrices, it suffices to impose that
\begin{equation}\label{u-opt-admi}
0\le \bar\mu(t)\le \frac{1}{2}(a+b).
\end{equation} 
                           
\no We further characterize the optimal control $\bar{u}(t)$ by finding $\bar\mu(t)$ which satisfies \eqref{u-opt-admi}. Indeed, under $P^{\bar{u}}$, $x$ has the representation
\begin{align*}
x(t) =& x(0) + \int_0^t \left\{ \a (b-a) I_a(s^-) + \bar{\mu}(s) (a-b) I_b(s^-) + \bar{u}(s) (a-b) I_b(s^-) \right\} ds \\
     &+ \int_0^t(b-a) dM^{\bar{u}}_{ba}(s) + \int_0^t (a-b) dM^{\bar{u}}_{ab}(s).
\end{align*}
Taking the expectation under $P^{\bar{u}} $ yields
 \begin{equation} \label{mu}
 \bar{\mu}(t)  = \bar{\mu} (0) + E^{\bar{u}} [\int_0^t \left\{  \a (b-a) I_a(s^-)+ \bar{\mu}(s) (a-b) I_b(s^-) + (a-b) \bar{u}(s) I_b(s^-) \right\} ds ].
 \end{equation}
In particular, the mapping $ t \rightarrow \mu(t) $ is absolutely continuous.            
\no Using the fact that $(a-b) I_b(t^-) = a-x(t^-) $ and $(b-a) I_a(t^-) = b-x(t^-)$, equation \eqref{mu}becomes 
\begin{align*}
\bar{\mu}(t) =& \bar{\mu}(0)+\int_0^t \left\{ \a (b-\bar{\mu}(s)) + \bar{\mu}(s)(a-\bar{\mu}(s)) \right\}ds + \int_0^t \left\{ E^{\bar{u}} [(a-b) \bar{u}(s) I_b(s^-) ] \right\} ds  \\
  =& \bar{\mu}(0) + \int_0^t \left\{ \a (b-\bar{\mu}(s)) + \bar{\mu}(s)(a-\bar{\mu}(s)) \right\}ds \label{last term-mu} \\
  &+ \int_0^t \left\{ E^{\bar{u}} \left[(a-b) I_b(s^-) \left( (b^2-a^2) +2 \bar{\mu}(s) (a-b) \right) \right] \right\}ds \\
  =&\bar{\mu}(0)+\int_0^t \left\{ \a (b-\bar{\mu}(s)) + \bar{\mu}(s)(a-\bar{\mu}(s)) \right\}ds \\
  &+ \int_0^t \left\{ (b^2-a^2) (a-\bar{\mu}(s)) + 2 (a-b) \bar{\mu}(s) (a-\bar{\mu}(s)) \right\} ds                                                                                \\
  =&\bar{\mu}(0)+\int_0^t \left\{ \left( \a b +a(b^2-a^2) \right)+ \left(2(b-a)-1\right) \bar{\mu}^2(s) + (3a^2+a(1-2b)-b^2) \bar{\mu}(s) \right\}ds,                                                                       
\end{align*}
\no with
\begin{equation*}
\left\{ \begin{array}{lll}
A  := 2(b-a)-1,   \\ 
B  := 3a^2+a(1-2b)-b^2,  \\
C  := \a b +a(b^2-a^2).
\end{array}
\right.
\end{equation*}
Thus, in view \eqref{u-opt-admi}, $\bar{\mu}$ should satisfy the following constrained Riccati equation  
\begin{equation}\label{ex-2-riccati}
\left\{ \begin{array}{lll}
\dot{\bar\mu}(t) = A {\bar\mu}^2(t) + B \bar{\mu}(t) + C,   \\ 
\bar\mu(0) = m_0,\\
0\le {\bar\mu}(t)\le \frac{1}{2}(a+b),
\end{array}
\right.
\end{equation}
where $m_0$ is a given initial value.   As it is well known, without the imposed constraint on $\bar\mu$, the Riccati equation admits an explicit solution that may explode in finite time unless the involved coefficients $a,b,\a$ and $m_0$ evolve within certain ranges. With the imposed constraint on $\bar\mu$, these ranges may become further tighter. Below we illustrate this through a few cases.  
As shown in the tables below, for low values of $\a$, the ODE \eqref{ex-2-riccati} can be solved for any time. How low the intensity should be mainly depends on the size of  $b$ and $b-a$, the larger is $b$ the wider is the range for $\a$ for which the ODE is solvable. In particular, when $a=0$ and $b=1$, \eqref{ex-2-riccati} is solvable for any time when $\alpha =0.1,0.2 $. For greater values of $\alpha $ the ODE violates the constraint proportionally "faster".    

The results also show that the initial conditions may affects the time horizon $T$. Starting with values reasonably close to   $\frac{a+b}{2}$  the ODE \eqref{ex-2-riccati}  is solvable only for relatively shorter time horizons than when we start with values reasonably close to zero. 
\\
\begin{center}
    \begin{tabular}{l*{5}{c}r}
    \hline
     a & b & $\alpha$ & $T_{m_0=0}$ & $ T_{m_0=0.25}$
     \\  \hline
     0 & 1 & 0.1  & .&.&								
     \\ \hline
     0 & 1   & 0.2 & .&.&
     \\ \hline
     0 & 1   & 0.3  & 5.145&3.762
    \\ \hline
     0 & 1   & 0.4 & 2.355&1.481
     \\ \hline
     0 & 1   & 0.5 & 1.571&0.928
    \\ \hline
     0 & 1   & 0.6 & 1.870&0.676
     \\ \hline
     0 & 1   & 0.7 & 0.955&0.532
     \\ \hline
     0 & 1   & 0.8 & 0.800&0.439
     \\ \hline
     0 & 1   & 0.9 & 0.689&0.373
     \\ \hline
     0 & 1   & 1   & 0.605&0.325
     \\ \hline
     0 & 1   & 5   & 0.104&0.053	
     \\ \hline
     0 & 1   & 10  & 0.051&0.026
    \end{tabular}
\quad
\begin{tabular}{l*{5}{c}r}
    \hline
     a & b & $\alpha$ &$ T_{m_0=0.25}$ & $T_{m_0=1}$
     \\  \hline
     1 & 2 & 0.1  & .&.&
     \\ \hline
     1 & 2 & 0.2 & .&.&
     \\ \hline
     1 & 2  & 0.3 & .&.&
    \\ \hline
     1 & 2   & 0.4  &2.644&2.153							
     \\ \hline
     1 & 2   & 0.5 &1.429& 1.001
    \\ \hline
     1 & 2   & 0.6 &1.073& 0.692
     \\ \hline
     1 & 2   & 0.7 &0.878& 0.535
     \\ \hline
     1 & 2   & 0.8 &0.750& 0.438
     \\ \hline
     1 & 2   & 0.9 &0.659& 0.371
     \\ \hline
     1 & 2   & 1 &0.589&  0.322
     \\ \hline
     1 & 2   & 5 &0.121&  0.053
     \\ \hline
     1 & 2   & 10 &0.062& 0.026
    \end{tabular}
    \end{center}
    
    \begin{center}
\begin{tabular}{l*{5}{c}r}
    \hline
     a & b & $\alpha$ & $ T_{m_0=0.25}$ & $T_{m_0=2}$ 
     \\  \hline
     2 & 3 & 0.1  & .&.&
     \\ \hline
     2 & 3 & 0.2 & .&.&
     \\ \hline
     2 & 3  & 0.3 & .&.& 
    \\ \hline
     2 & 3   & 0.4  &.&.
     \\ \hline
     2 & 3   & 0.5 &1.206& 0.761							
    \\ \hline
     2 & 3   & 0.6 & 0.899&0.494
     \\ \hline
     2 & 3   & 0.7 &0.746 &0.373
     \\ \hline
     2 & 3   & 0.8 &0.648& 0.302
     \\ \hline
     2 & 3   & 0.9 &0.578& 0.254
     \\ \hline
     2 & 3   & 1 &0.524&  0.220
     \\ \hline
     2 & 3   & 5 &0.131&  0.035
     \\ \hline
     2 & 3   & 10 &0.070& 0.018
    \end{tabular}
    \quad
\begin{tabular}{l*{5}{c}r}
    \hline
     a & b & $\alpha$ & $T_{m_0=0}$ & $ T_{m_0=0.75}$    
     \\  \hline
     0 & 2 & 0.1  & .&.&
     \\ \hline
     0 & 2 & 0.2 &  .&.&
     \\ \hline
     0 & 2  & 0.3 & .&.&
    \\ \hline
     0 & 2   & 0.4 & .&.&
     \\ \hline
     0 & 2   & 0.5 & .&.&
    \\ \hline
     0 & 2   & 0.6 & .&.&
     \\ \hline
     0 & 2   & 0.7 & 5.593&1.433			
     \\ \hline
     0 & 2   & 0.8 & 2.227&0.636
     \\ \hline
     0 & 2   & 0.9 & 1.470&0.418
     \\ \hline
     0 & 2   & 1 &  1.111&0.312
     \\ \hline
     0 & 2   & 5 &  0.112&0.029
     \\ \hline
     0 & 2   & 10 & 0.053&0.014
    \end{tabular}
    \end{center}
    
    \begin{center}
    \begin{tabular}{l*{5}{c}r}
    \hline
     a & b & $\alpha$ & $T_{m_0=0}$ & $ T_{m_0=1}$ 
     \\  \hline
     0 & 3 & 0.1  & .&.&
     \\ \hline
     0 & 3 & 0.2 &  .&.&
     \\ \hline
     0 & 3  & 0.3 & .&.&
    \\ \hline
     0 & 3   & 0.4 & .&.&
     \\ \hline
     0 & 3   & 0.5 & .&.&
    \\ \hline
     0 & 3   & 0.6 & .&.&
     \\ \hline
     0 & 3   & 0.7 & .&.&
     \\ \hline
     0 & 3   & 0.8 & .&.&
     \\ \hline
     0 & 3   & 0.9 & .&.&
     \\ \hline
     0 & 3   & 1 &  .&.&
     \\ \hline
     0 & 3   & 5 &  0.126&0.043	
     \\ \hline
     0 & 3   & 10 & 0.056&0.019
    \end{tabular}
  \end{center}


\subsection*{Example 3. Mean-field Schl\"ogl model}
\no We suggest to solve a control problem  associated with a mean-field  version of the Schl\"ogl model (cf. \cite{NP}, \cite{Chen}, \cite{DZ} and \cite{FZ}) where the intensities are of the form 
\begin{equation}\label{Sc-lambda}
 \l_{ij}^u(t,x,u(t)):=\left\{\begin{array}{ll} \nu_{ij}(t) & \text{if} \,\, j\neq i-1,\\
u(t) + \b E^{u}[x(t)]  & \text{if} \,\, j= i-1,
 \end{array}
 \right.
 \end{equation}
for some predictable and positive control process $u$,  where $ \b> 0 $ and
 $(\alpha_{ij})_{ij}$ is a deterministic $Q$-matrix for which there exists $N_0\ge 1$ such
that $\alpha_{ij}= 0$ for $|j-i|\ge N_0$ and $\alpha_{ij}> 0$ for $|j-i|< N_0$.

\no We consider the following mean field-type cost functional 

\begin{equation}\label{Sc-opt-u}
J(u)= E^u\left[ \int_0^T  \frac{1}{2}u^2(t)dt+x(T) \right]. 
\end{equation}

\no Given a control $v>0$, the associated Hamiltonian reads
$$ 
H(t,v):= L^{\bar{u}} \left( \underset{ i,j,j\neq i, i-1  }\sum \{ I_i(t^-)(\alpha_{ij}-g_{ij}) q_{ij}(t) \}+ \underset{ i }\sum \{ I_i(t^-)(v+\beta_1E^{\bar{u}}[x(t)]-g_{ii-1}) q_{ii-1}(t) \}-
 \frac{v^2}{2} \right). 
$$

\no The first-order optimality conditions yield 
\begin{equation}\label{Sc-opt-u}
\bar{u}(t)= \underset{ i }\sum I_i(t^-)q_{ii-1} (t)
\end{equation}

\no Next, we write the associated  adjoint equation and identify $q_{ii-1}.$
\begin{equation*}
\left\{ \begin{array}{lll}
dp(t)= -\underset{i}\sum q_{ii-1}(t)\left\{ I_i(t^-)(\frac{1}{2}q_{ii-1}(t) +\b E^{\bar{u}}[x(t)]-g_{ii-1})dt -dM_{ii-1}(t) \right\} \\ \qquad\qquad\qquad
      -\underset{ i j,j\neq i, i-1 }\sum q_{ij}(t) \left\{I_i(t^-)(\nu_{ij}-g_{ij}) dt - dM_{ij}(t)\right\} -\beta x(t)E^{\bar{u}}\left[\underset{i}\sum q_{ii-1}(t)I_i(t^-)\right]dt,  
\\
p(T)= - x(T) .
\end{array}
\right.
\end{equation*}

\no Consider the probability measure $\widetilde P $, under which $x$ is a pure jump process whose jump intensity matrix is

\begin{equation*}
\widetilde G_{ij}(t) =\left\{ \begin{array}{lll}
\alpha_{ij} & \text{if} \,\, j\neq i-1,
\\
\frac{1}{2} q_{ij}(t) + \b E^{\bar{u}}[x(t)]  & \text{if} \,\, j= i-1.
\end{array}
\right.
\end{equation*}

The adjoint equations becomes
\begin{equation*}
\left\{ \begin{array}{lll}
dp(t)=-\beta x(t) E^{\bar{u}}\underset{i}\sum [q_{ii-1}(t)I_i(t^-)] dt+ \underset{ i }\sum q_{ii-1}(t) d\widetilde M_{ii-1}(t)+ \underset{ i\neq j,j\neq i-1  }\sum q_{ij}(t) d\widetilde M_{ij}(t),  
\\
p(T)= - x(T),
\end{array}
\right.
\end{equation*}
where $\widetilde M_{ij},\,\, i\neq j$ are mutually orthogonal $\widetilde P $-martingales.

Thus
$$
p(t) = - \widetilde{E} [x(T)| \mathcal{F}_t ] + \b \int_t^T \widetilde{E} \left[ x(s) E^{\bar{u}}\left[\underset{i}\sum q_{ii-1}I_i(s^-)\right]\big| \mathcal{F}_t \right]ds,
$$
and
\begin{equation} \label{matching-schlogl}
\Delta p(t)  = \Delta \widetilde{E} [-x(T)| \mathcal{F}_t ].
\end{equation}

Following the same steps leading to \eqref{ex-1-q}, from \ref{matching-schlogl} we obtain $q_{ij}(t)=i-j$, thus $q_{ii-1}(t)=1,\,, i=1,2,\ldots $ 

Therefore, 
$$
\bar{u}(t)= \underset{ i\ge 1 }\sum I_i(t^-)=1-I_0(t^-).
$$


\end{document}